\documentclass[11pt]{article}
\usepackage[]{amsmath,amssymb}
\usepackage{cite}
\newtheorem{theorem}{Theorem}[section]
\newtheorem{lemma}[theorem]{Lemma}

\newtheorem{corollary}[theorem]{Corollary}
\newtheorem{exAux}[theorem]{Example}

\newtheorem{Def}[theorem]{Definition}
\newenvironment{definition}{\begin{Def} \rm}{\end{Def}}
\newtheorem{Note}[theorem]{Note}

\newtheorem{Problem}[theorem]{Problem}

\newtheorem{Rem}[theorem]{Remark}

\newtheorem{Not}[theorem]{Notation}

\newtheorem{Conj}[theorem]{Conjecture}

\newtheorem{Ass}[theorem]{Assumption}

\newenvironment{proof}{\medskip\noindent{\bf Proof.\ }}{\qed\medskip}
\newenvironment{proofof}[1]{\medskip\noindent{\bf Proof  of {#1}.\ 
}}{\qed\medskip}
\newcommand{\qed}{\hfill\mbox{$\Box$\qquad\qquad}}

\renewcommand{\th}{\theta}

\newcommand{\R}{\mathbb{R}}

\newcommand{\matR}{\text{\rm Mat}_{d+1}(\mathbb{R})}
\newcommand{\matX}{\text{\rm Mat}_X(\mathbb{R})}
\newcommand{\hx}{\hat{x}}
\newcommand{\bb}{\text{\bf 1}}
\renewcommand{\indent}{\hspace{6mm}}

\begin{document}

\thispagestyle{empty}

\begin{center}
\LARGE \bf
\noindent
Tridiagonal matrices with nonnegative entries
\end{center}

\smallskip

\begin{center}
\Large
Kazumasa Nomura and Paul Terwilliger
\end{center}

\smallskip

\begin{quote}
\small 
\begin{center}
\bf Abstract
\end{center}
\indent
In this paper we characterize the nonnegative irreducible tridiagonal matrices and 
their permutations, using certain entries in their primitive idempotents.
Our main result is summarized as follows. 
Let $d$ denote a nonnegative integer.
Let $A$ denote a matrix in $\matR$ and let $\{\th_i\}_{i=0}^d$
denote the roots of the characteristic polynomial of $A$.
We say $A$ is {\em multiplicity-free} whenever these 
roots are mutually distinct and contained in $\R$.
In this case $E_i$ will denote the primitive idempotent of $A$
associated with $\th_i$ $(0 \leq i \leq d)$.
We say $A$ is {\em symmetrizable} whenever there exists an invertible
diagonal matrix $\Delta \in \matR$ such that $\Delta A \Delta^{-1}$ 
is symmetric.
Let $\Gamma(A)$ denote the directed graph with vertex set $\{0,1,\ldots,d\}$,
where $i \rightarrow j$ whenever $i \neq j$ and $A_{ij} \neq 0$.

\noindent
{\bf Theorem}.
{\em
Assume that each entry of $A$ is nonnegative. 
Then the following are equivalent for $0 \leq s,t \leq d$.
\begin{itemize}
\item[\rm (i)]
The graph $\Gamma(A)$ is a bidirected path with endpoints
$s$, $t$:
\[
 s \leftrightarrow * \leftrightarrow * \leftrightarrow \cdots
 \leftrightarrow * \leftrightarrow t.
\]
\item[\rm (ii)]
The matrix $A$ is symmetrizable and multiplicity-free.
Moreover the $(s,t)$-entry of $E_i$ times
\[
 (\th_i-\th_0)\cdots(\th_i-\th_{i-1})
 (\th_i-\th_{i+1})\cdots(\th_i-\th_d)
\]
is independent of $i$ for $0 \leq i \leq d$, and this
common value is nonzero.
\end{itemize}
}

Recently Kurihara and Nozaki obtained a theorem that characterizes the 
$Q$-polynomial property for symmetric association schemes. 
We view the above result as a linear algebraic generalization
of their theorem.
\end{quote}

\section{Introduction}

\indent
Recently Kurihara and Nozaki gave the following characterization
of the $Q$-polynomial property for symmetric association schemes
(see Section \ref{sec:assoc} for definitions).

\medskip

\begin{theorem} {\rm \cite[Theorem 1.1]{KN}} \label{thm:Q} \samepage
Let ${\cal X}$ denote a $d$-class symmetric association scheme
with adjacency matrices $\{A_i\}_{i=0}^d$.
Let $E$ and $F$ denote primitive idempotents of $\cal X$ with $E$ nontrivial.
For $0 \leq i \leq d$ let $\th^*_i$ denote the dual eigenvalue
of $E$ for $A_i$.
Then the following are equivalent.
\begin{itemize}
\item[\rm (i)]
$\cal X$ is $Q$-polynomial relative to $E$, and $F$ is the
last primitive idempotent in this $Q$-polynomial structure. 
\item[\rm (ii)]
$\{\th^*_i\}_{i=0}^d$ are mutually distinct, and
for $0 \leq i \leq d$ the eigenvalue of $A_i$ for $F$ is
\begin{equation}    \label{eq:pid}
    \frac{(\th^*_0-\th^*_1)(\th^*_0-\th^*_2)\cdots(\th^*_0-\th^*_d)}
         {(\th^*_i-\th^*_0)\cdots(\th^*_i-\th^*_{i-1})
          (\th^*_i-\th^*_{i+1})\cdots(\th^*_i-\th^*_d)}.
\end{equation}
\end{itemize}
\end{theorem}

\medskip

As suggested by \cite{KN}, there is a ``dual'' version of Theorem \ref{thm:Q}
in which the $Q$-polynomial structure is replaced by a $P$-polynomial structure.
We now state this dual version.

\medskip

\begin{theorem}     \label{thm:P}    \samepage
Let ${\cal X}$ denote a $d$-class symmetric association scheme
with primitive idempotents $\{E_i\}_{i=0}^d$.
Let $B$ and $C$ denote adjacency matrices of $\cal X$ with $B$ nontrivial.
For $0 \leq i \leq d$ let $\th_i$ denote the eigenvalue of $B$ for $E_i$.
Then the following are equivalent.
\begin{itemize}
\item[\rm (i)]
$\cal X$ is $P$-polynomial relative to $B$, and $C$ is the
last adjacency matrix in this $P$-polynomial structure.
\item[(ii)]
$\{\th_i\}_{i=0}^d$ are mutually distinct, and
for $0 \leq i \leq d$ the dual eigenvalue of $E_i$ for $C$ is
\begin{equation}    \label{eq:qid}
    \frac{(\th_0-\th_1)(\th_0-\th_2)\cdots(\th_0-\th_d)}
         {(\th_i-\th_0)\cdots(\th_i-\th_{i-1})
          (\th_i-\th_{i+1})\cdots(\th_i-\th_d)}.
\end{equation}
\end{itemize}
\end{theorem}

\medskip

In this paper we show that Theorems \ref{thm:Q} and \ref{thm:P} 
follow from a linear algebraic result concerning matrices with nonnegative 
entries.
We give two versions of the result, which are Theorems \ref{thm:mainsym}
and \ref{thm:main} below. Theorem \ref{thm:main} is the general version, 
and Theorem \ref{thm:mainsym} is about an attractive special case. 
Before presenting these theorems we recall some concepts from linear
algebra.

\medskip

Throughout the paper $\R$ denotes the field of real numbers,
$d$ denotes a nonnegative integer, and
$\matR$ denotes the $\R$-algebra consisting of the $(d+1)\times(d+1)$ 
matrices that have all entries in $\R$.
We index the rows and columns by $0,1,\ldots,d$.
Let $V=\R^{d+1}$ denote the vector space over $\R$ consisting of the
$(d+1) \times 1$ matrices that have all entries in $\R$.
We index the rows by $0,1,\ldots, d$.
Observe that $\matR$ acts on $V$ by left multiplication.

Let $A$ denote a matrix in $\matR$.
We say $A$ is {\em nonnegative} whenever each entry of $A$ is nonnegative.
We say $A$ is {\em symmetric} whenever $A^t=A$, where $t$ denotes
transpose.
We say $A$ is {\em symmetrizable} whenever there exists an invertible
diagonal matrix $\Delta \in \matR$ such that $\Delta A \Delta^{-1}$ is symmetric.
A subspace $W \subseteq V$ is called an {\em eigenspace} of $A$ whenever 
$W \neq 0$ and there exists $\th \in \R$ such that
$W=\{v \in V \,|\, Av = \th v\}$;
in this case $\th$ is the {\em eigenvalue} of $A$ associated with $W$.
We say $A$ is {\em diagonalizable} whenever its eigenspaces span $V$.
We say that $A$ is {\em multiplicity-free} whenever
$A$ is diagonalizable and its eigenspaces all have dimension $1$.
Assume $A$ is multiplicity-free and let $\{\th_i\}_{i=0}^d$ denote
an ordering of the eigenvalues of $A$.
For $0 \leq i \leq d$ let $V_i$ denote 
the eigenspace of $A$ associated with $\th_i$.
For $0 \leq i \leq d$ define $E_i \in \matR$ such that
$(E_i-I)V_i=0$ and $E_iV_j=0$ for $j \neq i$ $(0 \leq j \leq d)$.
Here $I$ denotes the identity matrix in $\matR$.
We call $E_i$ the {\em primitive idempotent} of $A$ associated with $V_i$
(or $\th_i$). 
Observe that
(i) $I=\sum_{i=0}^d E_i$;
(ii) $E_iE_j=\delta_{i,j}E_i$ $(0 \leq i,j \leq d)$;
(iii) $V_i=E_iV$ $(0 \leq i \leq d)$;
(iv) $A=\sum_{i=0}^d \th_i E_i$.
Using these facts we find
\begin{align}    \label{eq:defEi}
 E_i &= \prod_{\stackrel{0 \leq j \leq d}{j \neq i}}
        \frac{A-\th_jI}
             {\th_i-\th_j}
     & & (0 \leq i \leq d).
\end{align}

Again let $A$ denote a matrix in $\matR$.
We say that $A$ is {\em tridiagonal} whenever each nonzero entry 
lies on either the diagonal, the subdiagonal, or the superdiagonal.
Assume for the moment that $A$ is tridiagonal.
Then $A$ is said to be {\em irreducible} whenever each entry on the 
subdiagonal is nonzero and each entry on the superdiagonal is nonzero.
Let $\Gamma(A)$ denote the directed graph with 
vertex set $\{0,1,\ldots,d\}$, where $i \rightarrow j$ whenever 
$i \neq j$ and $A_{ij} \neq 0$.
Observe that the following are equivalent:
(i) $A$ is irreducible tridiagonal;
(ii) $\Gamma(A)$ is the bidirected path
$0 \leftrightarrow 1 \leftrightarrow 2 \leftrightarrow \cdots
  \leftrightarrow d$.
More generally the following are equivalent:
(i) there exists a permutation matrix $\Lambda \in \matR$ such that
 $\Lambda A\Lambda^{-1}$ is irreducible tridiagonal;
(ii) $\Gamma(A)$ is a bidirected path.
We now state our first main result.

\medskip

\begin{theorem}    \label{thm:mainsym}    \samepage
Let $A$ denote a nonnegative matrix in $\matR$.
Then the following are equivalent for $0 \leq s,t \leq d$.
\begin{itemize}
\item[\rm (i)]
The graph $\Gamma(A)$ is a bidirected path with endpoints
$s$, $t$:
\[
 s \leftrightarrow * \leftrightarrow * \leftrightarrow \cdots
 \leftrightarrow * \leftrightarrow t.
\]
\item[\rm (ii)]
The matrix $A$ is symmetrizable and multiplicity-free.
Moreover the $(s,t)$-entry of $E_i$ times
\[
   (\th_i-\th_0)\cdots(\th_i-\th_{i-1})
   (\th_i-\th_{i+1})\cdots(\th_i-\th_d)
\]
is independent of $i$ for $0 \leq i \leq d$, 
and this common value is nonzero.
\end{itemize}
\end{theorem}

\medskip

The proof of Theorem \ref{thm:mainsym} is given in Section
\ref{sec:mainsymproof}.
Before stating our second main result, we make a few comments.
For $A \in \matR$ we say that $A$ is {\em (upper) Hessenberg} whenever 
each entry below the
subdiagonal is zero and each entry on the subdiagonal is nonzero.
Observe that the following are equivalent:
(i) $A$ is Hessenberg;
(ii) in the graph $\Gamma(A)$, for all vertices $i,j$ we have
$i \rightarrow j$ if $i-j=1$ and $i \not\rightarrow j$ if $i-j>1$.
An ordering $\{x_i\}_{i=0}^d$ of the vertices of $\Gamma(A)$ is called 
{\em Hessenberg} whenever for $0 \leq i,j\leq d$,
$x_i \rightarrow x_j$ if $i-j=1$ and $x_i \not\rightarrow x_j$ if $i-j>1$.
We recall the directed distance function $\partial$ for $\Gamma(A)$.
Given vertices $s,t$ of $\Gamma(A)$ and an integer $i$ $(0 \leq i \leq d)$,
we have $\partial(s,t)=i$ whenever there exists a directed path in $\Gamma(A)$
from $s$ to $t$ that has length $i$, and there does not exist a directed path 
in $\Gamma(A)$ from $s$ to $t$ that has length less than $i$.
For all vertices $s,t$ in $\Gamma(A)$ the following are equivalent:
(i) there exists a Hessenberg ordering $\{x_i\}_{i=0}^d$ of the vertices of 
$\Gamma(A)$ such that $x_0=t$ and $x_d=s$;
(ii) $\partial(s,t)=d$.
We now state our second main result.

\medskip

\begin{theorem}            \label{thm:main}     \samepage
Let $A$ denote a nonnegative matrix in $\matR$.
Then the following are equivalent for $0 \leq s,t \leq d$.
\begin{itemize}
\item[\rm (i)]
The matrix $A$ is diagonalizable, and $\partial(s,t)=d$ in $\Gamma(A)$.
\item[\rm (ii)]
The matrix $A$ is multiplicity-free.
Moreover the $(s,t)$-entry of $E_i$ times
\[
   (\th_i-\th_0)\cdots(\th_i-\th_{i-1})
   (\th_i-\th_{i+1})\cdots(\th_i-\th_d)
\]
is independent of $i$ for $0 \leq i \leq d$, 
and this common value is nonzero.
\end{itemize}
\end{theorem}

\medskip

The proof of Theorem \ref{thm:main} is given in Section \ref{sec:mainproof}.
In Sections \ref{sec:assoc}--\ref{sec:PQpoly} we apply Theorem \ref{thm:mainsym}
to symmetric association schemes.
In Sections \ref{sec:assoc} and \ref{sec:rho} we give some basic facts about 
these objects.
In Section \ref{sec:PQpoly} we use these facts and Theorem \ref{thm:mainsym}
to prove Theorems \ref{thm:Q} and \ref{thm:P}.

\section{Hessenberg matrices}
\label{sec:mainproof}

\indent
In this section we prove Theorem \ref{thm:main}.

\medskip

\begin{lemma}     \label{lem:Ari0}
Let $A$ denote a Hessenberg matrix in $\matR$.
Then for $0 \leq r \leq d$ the entries of $A^r$ are described as
follows. 
For $0 \leq i,j\leq d$ the $(i,j)$-entry is nonzero if
$i-j=r$ and zero if $i-j>r$.
\end{lemma}

\begin{proof}
Use matrix multiplication and the definition of Hessenberg.
\end{proof}

\begin{corollary}    \label{cor:Arindep}
Let $A$ denote a Hessenberg matrix in $\matR$. 
Then the matrices $\{A^r\}_{r=0}^d$ are linearly independent.
\end{corollary}

\begin{proof}
For $0 \leq r \leq d$ let $u_r \in \R^{d+1}$ denote the $0^\text{th}$
column of $A^r$.
By Lemma \ref{lem:Ari0} the $i^\text{th}$ entry of $u_r$ is nonzero
for $i=r$ and zero for $r+1 \leq i \leq d$.
Therefore $\{u_r\}_{r=0}^d$ are linearly independent.
The result follows.
\end{proof}

\begin{lemma}   \label{lem:Hmin}
Let $A$ denote a Hessenberg matrix in $\matR$. 
Then the minimal polynomial of $A$ equals the characteristic 
polynomial of $A$.
\end{lemma}

\begin{proof}
By construction the characteristic polynomial of $A$ is monic 
with degree $d+1$.
By elementary linear algebra the minimal polynomial of $A$ is monic and 
divides the characteristic polynomial of $A$. 
By Corollary \ref{cor:Arindep} the minimal polynomial of $A$ has degree 
$d+1$. 
The result follows.
\end{proof}

\begin{lemma}    \label{lem:Hdm}    \samepage
Let $A$ denote a diagonalizable Hessenberg matrix in $\matR$.
Then $A$ is multiplicity-free.
\end{lemma}

\begin{proof}
Let $\{\th_i\}_{i=0}^d$ denote the roots of the characteristic
polynomial of $A$.
We have $\th_i \in \R$ $(0 \leq i \leq d)$ since $A$ is diagonalizable.
Moreover $\{\th_i\}_{i=0}^d$ are mutually distinct since the
minimal polynomial of $A$ equals the characteristic polynomial of $A$
by Lemma \ref{lem:Hmin} and since the roots of the minimal polynomial
are mutually distinct.
Thus $A$ is multiplicity-free.
\end{proof}

\medskip

For $A \in \matR$ let $\Gamma_\ell(A)$ denote the directed graph with
vertex set $\{0,1,\ldots,d\}$, where $i \rightarrow j$ whenever 
$A_{ij}\neq 0$.

\medskip

\begin{lemma}    \label{lem:premain}   \samepage
Let $A$ denote a nonnegative matrix in $\matR$.
Then the following are equivalent for $0 \leq r,s,t \leq d$.
\begin{itemize}
\item[\rm (i)]
The $(s,t)$-entry of $A^r$ is nonzero.
\item[\rm (ii)]
In the graph $\Gamma_\ell(A)$ there exists a directed path of length 
$r$ from $s$ to $t$.
\end{itemize}
\end{lemma}

\begin{proof}
Consider the $(s,t)$-entry of $A^r$ using matrix multiplication.
\end{proof}

\begin{lemma}    \label{lem:partial}  \samepage
The following are equivalent for all $A \in \matR$ and 
$0 \leq r,s,t \leq d$.
\begin{itemize}
\item[\rm (i)]
$\partial(s,t)=r$ in $\Gamma(A)$.
\item[\rm (ii)]
$\partial(s,t)=r$ in $\Gamma_\ell(A)$.
\end{itemize}
\end{lemma}

\begin{proof}
Routine verification.
\end{proof}

\begin{lemma}            \label{lem:main}     \samepage
Let $A$ denote a nonnegative matrix in $\matR$.
Then the following are equivalent for $0 \leq s,t \leq d$.
\begin{itemize}
\item[\rm (i)]
The $(s,t)$-entry of $A^r$ is nonzero if $r=d$ and zero if $r<d$
$(0 \leq r \leq d)$.
\item[\rm (ii)]
$\partial(s,t)=d$ in $\Gamma(A)$.
\end{itemize}
\end{lemma}

\begin{proof}
Follows from Lemmas \ref{lem:premain} and \ref{lem:partial}
\end{proof}

\medskip

Let $\lambda$ denote an indeterminate and let $\R[\lambda]$ denote the
$\R$-algebra consisting of the polynomials in $\lambda$ that have
all coefficients in $\R$.

\medskip

\begin{lemma}    \label{lem:Hpre}       \samepage
For $0 \leq i \leq d$ let $f_i \in \R[\lambda]$ be monic with degree $d$,
and assume $\{f_i\}_{i=0}^d$ are linearly independent.
Then the following are equivalent for all $A \in \matR$ and 
$0 \leq s,t \leq d$.
\begin{itemize}
\item[\rm (i)]
The $(s,t)$-entry of $A^r$ is zero for $0 \leq r \leq d-1$.
\item[\rm (ii)]
The $(s,t)$-entry of $f_i(A)$ is equal to the $(s,t)$-entry of $A^d$ 
for $0 \leq i \leq d$.
\item[\rm (iii)]
The $(s,t)$-entry of $f_i(A)$ is independent of $i$ for $0 \leq i \leq d$.
\end{itemize}
\end{lemma}

\begin{proof}
(i)$\Rightarrow$(ii):
Since $f_i$ is monic with degree $d$.

(ii)$\Rightarrow$(iii):
Clear.

(iii)$\Rightarrow$(i):
For $1 \leq i \leq d$ define $g_i=f_i-f_0$ and observe that $g_i$ has degree 
at most $d-1$.
Note that $\{g_i\}_{i=1}^d$ are linearly independent. 
So $\{g_i\}_{i=1}^d$ form a basis for the subspace of 
$\R[\lambda]$ consisting of the polynomials with degree at most $d-1$.
So for $0 \leq r \leq d-1$, 
$\lambda^r$ is a linear combination of $\{g_i\}_{i=1}^d$.
By construction the $(s,t)$-entry of $g_i(A)$ is zero for $1 \leq i \leq d$.
By these comments the $(s,t)$-entry of $A^r$ is zero for $0 \leq r \leq d-1$.
\end{proof}

\medskip

Referring to Lemma \ref{lem:Hpre} we now make a specific choice for the
polynomials $\{f_i\}_{i=0}^d$.

\medskip

\begin{lemma}    \label{lem:H}    \samepage
Assume $A \in \matR$ is multiplicity-free with eigenvalues $\{\th_i\}_{i=0}^d$.
For $0 \leq i \leq d$ define a polynomial $f_i \in \R[\lambda]$ by
\begin{equation}    \label{eq:fi}
 f_i=(\lambda-\th_0)\cdots(\lambda-\th_{i-1})
     (\lambda-\th_{i+1})\cdots(\lambda-\th_d).
\end{equation}
Then
\begin{itemize}
\item[\rm (i)]
$f_i(A)=f_i(\th_i)E_i$ for $0 \leq i \leq d$.
\item[\rm (ii)]
$f_i$ is monic with degree $d$ for $0 \leq i \leq d$.
\item[\rm (iii)]
$\{f_i\}_{i=0}^d$ are linearly independent.
\end{itemize}
\end{lemma}

\begin{proof}
(i):
Compare \eqref{eq:defEi} and \eqref{eq:fi}.

(ii):
Clear.

(iii):
For $0 \leq i,j \leq d$ the scalar $f_i(\th_j)$ is zero if
$i\not=j$ and nonzero if $i=j$. 
\end{proof}

\begin{proofof}{Theorem \ref{thm:main}}

(i)$\Rightarrow$(ii):
By the comments above Theorem \ref{thm:main} there exists a Hessenberg
ordering $\{x_i\}_{i=0}^d$ of the vertices of $\Gamma(A)$ such that 
$x_0=t$ and $x_d=s$. 
Let $\Lambda \in \matR$ denote the permutation matrix that corresponds to the 
permutation $i \mapsto x_i$ $(0 \leq i \leq d)$. 
Then $\Lambda A \Lambda^{-1}$ is Hessenberg.
We assume $A$ is diagonalizable so $\Lambda A \Lambda^{-1}$ is diagonalizable. 
Now $\Lambda A \Lambda^{-1}$ is multiplicity-free by Lemma \ref{lem:Hdm} so
$A$ is multiplicity-free. 
Define the polynomials $\{f_i\}_{i=0}^d$ as in Lemma \ref{lem:H}.
By Lemma \ref{lem:main}, for $0 \leq r \leq d$
the $(s,t)$-entry of $A^r$ is nonzero if $r=d$
and zero if $r< d$.
By this and Lemma \ref{lem:Hpre}, 
the $(s,t)$-entry of $f_i(A)$ is independent of $i$ for $0 \leq i \leq d$,
and this common value is nonzero.
By this and Lemma \ref{lem:H}(i), the $(s,t)$-entry of $E_i$ times
$f_i(\th_i)$ is independent of $i$ for $0 \leq i \leq d$, and this
common value is nonzero.

(ii)$\Rightarrow$(i):
The matrix $A$ is diagonalizable since it is multiplicity-free.
Define $\{f_i\}_{i=0}^d$ as in Lemma \ref{lem:H}.
By assumption,
the $(s,t)$-entry of $E_i$ times $f_i(\th_i)$
is independent of $i$ for $0 \leq i \leq d$, and this common
value is nonzero.
By this and Lemma \ref{lem:H}(i),
the $(s,t)$-entry of $f_i(A)$ is independent of $i$ for $0 \leq i \leq d$,
and this common value is nonzero.
By this and Lemma \ref{lem:Hpre}, for $0 \leq r \leq d$
the $(s,t)$-entry of $A^r$ is nonzero if $r=d$ and zero if $r< d$.
By this and Lemma \ref{lem:main} we find $\partial(s,t)=d$.
\end{proofof}

\section{Tridiagonal matrices}
\label{sec:mainsymproof}

\indent
In this section we prove Theorem \ref{thm:mainsym}.

\medskip

\begin{lemma}    \label{lem:AijAji}
Let $A$ denote a symmetrizable matrix in $\matR$.
Then $A_{ij}=0$ if and only if
$A_{ji}=0$ $(0 \leq i,j \leq d)$.
\end{lemma}

\begin{proof}
Since $A$ is symmetrizable, there exists an invertible diagonal matrix 
$\Delta \in \matR$ such that $\Delta A \Delta^{-1}$ is symmetric.
Comparing the $(i,j)$-entry and the $(j,i)$-entry of 
$\Delta A \Delta^{-1}$ we find
$\Delta_{ii}A_{ij}\Delta_{jj}^{-1} = \Delta_{jj}A_{ji}\Delta_{ii}^{-1}$.
The result follows.
\end{proof}

\begin{lemma}    \label{lem:symdiag}     \samepage
Let $A$ denote a symmetrizable matrix in $\matR$.
Then $A$ is diagonalizable.
\end{lemma}

\begin{proof}
Since $A$ is symmetrizable, there exists an invertible diagonal
matrix $\Delta \in \matR$ such that  $\Delta A\Delta^{-1}$ symmetric.
By \cite[Corollary 3.3.1]{Serre} every symmetric matrix in $\matR$
is diagonalizable.
Therefore  $\Delta A\Delta^{-1}$ is diagonalizable, so $A$ is
diagonalizable.
\end{proof}

\begin{lemma}   \label{lem:psym}    \samepage
Let $A \in \matR$ denote a symmetrizable matrix.
Then $\Lambda A\Lambda^{-1}$ is symmetrizable for every permutation matrix 
$\Lambda \in \matR$.
\end{lemma}

\begin{proof}
Since $A$ is symmetrizable, there exists an invertible diagonal matrix
$\Delta \in \matR$ such that $\Delta A\Delta^{-1}$ is symmetric.
Set $\Delta'=\Lambda\Delta \Lambda^{-1}$, and 
observe that $\Delta'$ is invertible diagonal.
Using $\Lambda^{-1}=\Lambda^t$ we find
$\Delta' \Lambda A\Lambda^{-1}(\Delta')^{-1}$ is symmetric.
Now $\Lambda A\Lambda^{-1}$ is symmetrizable.
\end{proof}

\begin{lemma}      \label{lem:sym}     \samepage
Let $A$ denote a nonnegative irreducible tridiagonal matrix in $\matR$.
Then $A$ is symmetrizable and multiplicity-free.
\end{lemma}

\begin{proof}
We first show that $A$ is symmetrizable.
Since $A$ is irreducible and nonnegative we have
$A_{i,i-1}> 0$ and $A_{i-1,i}>0$ for $1 \leq i \leq d$.
For $0 \leq i \leq d$ define
\[
    \kappa_i = \frac{A_{01}A_{12} \cdots A_{i-1,i}}
               {A_{10}A_{21} \cdots A_{i,i-1}}
\]
and note that $\kappa_i > 0$.
Define a diagonal matrix $K \in \matR$ with $(i,i)$-entry $\kappa_i$
for $0 \leq i \leq d$.
Using matrix multiplication one finds $KA=A^tK$.
Define a diagonal matrix $\Delta \in \matR$ with $(i,i)$-entry
$\sqrt{\kappa_i}$ for $0 \leq i \leq d$, so that $\Delta^2=K$. 
By this and $KA=A^tK$ one finds that $\Delta A \Delta^{-1}$ is symmetric.
Therefore $A$ is symmetrizable.
Now $A$ is diagonalizable by Lemma \ref{lem:symdiag}
and multiplicity-free by Lemma \ref{lem:Hdm}.
\end{proof}

\begin{lemma}    \label{lem:mainsym}    \samepage
Let $A$ denote a nonnegative matrix in $\matR$.
Then the following are equivalent for $0 \leq s,t \leq d$.
\begin{itemize}
\item[\rm (i)]
The graph $\Gamma(A)$ is a bidirected path with endpoints $s$, $t$.
\item[\rm (ii)]
The matrix $A$ is symmetrizable, 
and $\partial(s,t)=d$ in $\Gamma(A)$.
\end{itemize}
\end{lemma}

\begin{proof}
(i)$\Rightarrow$(ii):
We first show that $A$ is symmetrizable.
By the observation above Theorem \ref{thm:mainsym},
there exists a permutation matrix $\Lambda \in \matR$ such that 
$\Lambda A \Lambda^{-1}$ is irreducible tridiagonal.
We assume $A$ is nonnegative so $\Lambda A \Lambda^{-1}$ is nonnegative.
So $\Lambda A \Lambda^{-1}$ is symmetrizable in vew of  Lemma \ref{lem:sym}.
Now $A$ is symmetrizable by Lemma \ref{lem:psym}.
By construction $\partial(s,t)=d$ in $\Gamma(A)$.

(ii)$\Rightarrow$(i):
Routine using Lemma \ref{lem:AijAji}.
\end{proof}

\begin{proofof}{Theorem \ref{thm:mainsym}}
(i)$\Rightarrow$(ii):
$A$ is symmetrizable by Lemma \ref{lem:mainsym}, and $A$ is
diagonalizable by Lemma \ref{lem:symdiag}.
By Lemma \ref{lem:mainsym}, $\partial(s,t)=d$ in $\Gamma(A)$.
The result follows in view of Theorem \ref{thm:main}.

(ii)$\Rightarrow$(i):
By Theorem \ref{thm:main} $\partial(s,t)=d$ in $\Gamma(A)$.
By this and Lemma \ref{lem:mainsym} the graph $\Gamma(A)$ is
a bidirected path with endpoints $s$, $t$.
\end{proofof}

\section{Symmetric association schemes}      \label{sec:assoc}

\indent
In this section we review some definitions and basic concepts concerning 
symmetric association schemes.
For more information we refer the reader to \cite{BI,BCN,T:subconst1}.

A {\em $d$-class symmetric association scheme} is a pair
${\cal X} = (X,\{R_i\}_{i=0}^d)$,
where $X$ is a finite nonempty set and $\{R_i\}_{i=0}^d$ are 
nonempty subsets of $X \times X$ that satisfy
\begin{itemize}
\item[(i)]
$R_0 = \{(x,x) \,|\, x \in X\}$;
\item[(ii)]
$X \times X = R_0 \cup R_1 \cup \cdots \cup R_d\;\;$ (disjoint union);
\item[(iii)]
$R_i^t=R_i$ for $0 \leq i \leq d$, where 
$R_i^t = \{(y,x) \,|\, (x,y) \in R_i\}$;
\item[(iv)]
there exist integers $p^h_{ij}$ $(0 \leq h,i,j \leq d)$ such that,
for every $(x,y) \in R_h$,
\[
 p^h_{ij}=|\{z \in X \,|\, (x,z) \in R_i, \; (z,y) \in R_j\}|.
\]
\end{itemize}
The parameters $p^h_{ij}$ are called the {\em intersection numbers}
of $\cal X$. 

From now on let ${\cal X}=(X, \{R_i\}_{i=0}^d)$ denote a 
$d$-class symmetric association scheme.
Observe by (iii) that $p^h_{ij}=p^h_{ji}$ for $0 \leq h,i,j \leq d$.
For $0 \leq i \leq d$ define $k_i = p^0_{ii}$, and observe
\begin{align*}
  k_i &= |\{y \in X \,|\, (x,y) \in R_i\}|
   & & (x \in X).
\end{align*}
Note that $k_i>0$.
By \cite[Proposition II.2.2]{BI},
\begin{align}   \label{eq:khphij}
 k_h p^h_{ij} &= k_j p^j_{ih}  & & (0 \leq h,i,j \leq d).
\end{align}

We recall the Bose-Mesner algebra of $\cal X$.
Let $\matX$ denote the $\R$-algebra consisting of the matrices whose 
rows and columns are indexed by $X$ and whose entries are in $\R$.
For $0 \leq i \leq d$ let $A_i$ denote the matrix in $\matX$ 
with $(x,y)$-entry
\begin{align*}
  (A_i)_{x,y} &=
   \begin{cases}
     1 & \text{ if $(x,y) \in R_i$} \\
     0 & \text{ if $(x,y) \not\in R_i$}
   \end{cases}
   & & (x,y \in X).
\end{align*}
We call $\{A_i\}_{i=0}^d$ the {\em adjacency matrices} of $\cal X$.
Note that $A_0 = I$, where $I$ denotes the identity matrix in $\matX$.
We call $A_0$ the {\em trivial} adjacency matrix.
Observe $A_i^t = A_i$ for $0 \leq i \leq d$. 
The matrices $\{A_i\}_{i=0}^d$ are linearly independent
since they have nonzero entries which are in disjoint positions.
Observe
\begin{align}
  A_iA_j &= \sum_{h=0}^d p^h_{ij} A_h
     & &     (0 \leq i,j \leq d).         \label{eq:AiAj}
\end{align}
By $p^h_{ij}=p^h_{ji}$ we find $A_iA_j = A_jA_i$ for $0 \leq i,j \leq d$.
Using these facts we find $\{A_i\}_{i=0}^d$ is a basis for a
commutative subalgebra $M$ of $\matX$.
We call $M$ the {\em Bose-Mesner algebra} of $\cal X$.

By \cite[Section II.2.3]{BI} $M$ has a second basis
$\{E_i\}_{i=0}^d$ such that
(i) $E_0 = |X|^{-1}J$;
(ii) $I=\sum_{i=0}^d E_i$;
(iii) $E_i^t = E_i$ $(0 \leq i \leq d)$;
(iv) $E_iE_j = \delta_{i,j}E_i$ $(0 \leq i,j \leq d)$.
We call $\{E_i\}_{i=0}^d$ the {\em primitive idempotents} of $\cal X$.
We call $E_0$ the {\em trivial} primitive idempotent.
For $0 \leq i \leq d$ let $m_i$ denote the rank of $E_i$.
Note that $m_i>0$.

We recall the matrices $P$ and $Q$.
We mentioned above that $\{A_i\}_{i=0}^d$ and $\{E_i\}_{i=0}^d$ are
bases for $M$. 
Define $P \in \matR$ such that
\begin{align}   
 A_j &= \sum_{i=0}^d P_{ij} E_i & & (0 \leq j \leq d). \label{eq:Ai}
\end{align}
Define $Q \in \matR$ such that
\begin{align}
 E_j &= |X|^{-1} \sum_{i=0}^d Q_{ij} A_i
                               & & (0 \leq j \leq d). \label{eq:Ei}
\end{align}
Observe that $PQ=QP = |X|I$. 
Setting $j=0$ and $A_0=I$ in \eqref{eq:Ai} we find $P_{i0}=1$
for $0 \leq i \leq d$.
Setting $j=0$ and $E_0 = |X|^{-1}J$ in \eqref{eq:Ei} we find
$Q_{i0}=1$ for $0 \leq i \leq d$.

We recall the $P$-polynomial property.
Let $\{A_i\}_{i=1}^d$ denote an ordering of the nontrivial adjacency
matrices of $\cal X$.
This ordering is said to be {\em $P$-polynomial} whenever
for $0 \leq i,j \leq d$ the intersection number $p^1_{ij}$ is zero
if $|i-j|>1$ and nonzero if $|i-j|=1$.
Let $A$ denote a nontrivial adjacency matrix of $\cal X$.
We say $\cal X$ is {\em $P$-polynomial relative to $A$}
whenever there exists a $P$-polynomial ordering $\{A_i\}_{i=1}^d$ of
the nontrivial adjacency matrices such that $A_1=A$.
In this case we call $A_d$ the {\em last adjacency matrix} in this
$P$-polynomial structure.

We recall the Krein parameters.
Let $\circ$ denote the entrywise product in $\matX$.
Observe $A_i \circ A_j = \delta_{i,j}A_i$ for $0 \leq i,j \leq d$,
so $M$ is closed under $\circ$.
Thus there exist
$q^h_{ij} \in \R$ $(0 \leq h,i,j \leq d)$ such that 
\begin{align}  
  E_i \circ E_j &= |X|^{-1} \sum_{h=0}^d q^h_{ij} E_h
                & & (0 \leq i,j \leq d).           \label{eq:EicircEj}
\end{align}
The parameters $q^h_{ij}$ are called the {\em Krein parameters}
of $\cal X$.
By \cite[Theorem II.3.8]{BI} the Krein parameters are nonnegative.
By \eqref{eq:EicircEj} we have $q^h_{ij}=q^h_{ji}$ for $0 \leq h,i,j \leq d$.
Setting $j=0$ and $E_0=|X|^{-1}J$ in \eqref{eq:EicircEj} we find
$q^h_{i0} = \delta_{h,i}$ for $0 \leq h,i \leq d$.
By \cite[Proposition II.3.7]{BI},
\begin{align}   \label{eq:mhqhij}
 m_h q^h_{ij} &= m_j q^j_{ih}  & & (0 \leq h,i,j \leq d).
\end{align}

We recall the $Q$-polynomial property.
Let $\{E_i\}_{i=1}^d$ denote an ordering of the nontrivial primitive
idempotents of $\cal X$.
This ordering is said to be {\em $Q$-polynomial} whenever
for $0 \leq i,j \leq d$ the Krein parameter $q^1_{ij}$ is zero if
$|i-j|>1$ and nonzero if $|i-j|=1$.
Let $E$ denote a nontrivial primitive idempotent of $\cal X$.
We say $\cal X$ is {\em $Q$-polynomial relative to $E$}
whenever there exists a $Q$-polynomial ordering $\{E_i\}_{i=1}^d$ of
the nontrivial primitive idempotents such that $E_1=E$.
In this case we call $E_d$ the {\em last primitive idempotent} in this
$Q$-polynomial structure.

We recall the dual Bose-Mesner algebra.
For the rest of the paper fix $x \in X$.
For $0 \leq i \leq d$ let $E^*_i$ denote the diagonal matrix in $\matX$
with $(y,y)$-entry
\begin{align*}
 (E^*_i)_{yy} &=
  \begin{cases}
   1 & \text{ if $(x,y) \in R_i$}  \\
   0 & \text{ if $(x,y) \not\in R_i$}
  \end{cases}
   & & (y \in X).
\end{align*}
For $y \in X$ the $(y,y)$-entry of $E^*_i$ coincides with the
$(x,y)$-entry of $A_i$.
Observe
$E^*_i E^*_j = \delta_{i,j} E^*_i$ $(0 \leq i,j \leq d)$ and
$I=\sum_{i=0}^d E^*_i$.  
We call $\{E^*_i\}_{i=0}^d$ the {\em dual primitive idempotents}
of $\cal X$.
By the above comments $\{E^*_i\}_{i=0}^d$ is a basis for a
commutative subalgebra ${M}^*$ of $\matX$.
We call ${M}^*$ the {\em dual Bose-Mesner algebra} of $\cal X$.
For $0 \leq i \leq d$ let $A^*_i$ denote the diagonal matrix in $\matX$
with $(y,y)$-entry $|X|(E_i)_{x,y}$ for $y \in X$.
We call $\{A^*_i\}_{i=0}^d$ the {\em dual adjacency matrices} of $\cal X$.
Using \eqref{eq:Ei},
\begin{align}
 A^*_j &= \sum_{i=0}^d Q_{ij} E^*_i  
             & & (0\leq j \leq d).                \label{eq:Asi}
\end{align}
Using \eqref{eq:Ai},
\begin{align}
 E^*_j &= |X|^{-1} \sum_{i=0}^d P_{ij} A^*_i
             & & (0 \leq j \leq d).              \label{eq:Esi}
\end{align}
Using \eqref{eq:EicircEj},
\begin{align}
 A^*_iA^*_j &= \sum_{h=0}^d q^h_{ij} A^*_h  &  
             & (0 \leq i,j \leq d).          \label{eq:AsiAsj}
\end{align}

For $0 \leq i,j \leq d$ the scalar $P_{ij}$ (resp. $Q_{ij}$)
is known as the {\em eigenvalue} of $A_j$ for $E_i$
(resp. {\em dual eigenvalue} of $E_j$ for $A_i$).

\section{The subconstituent algebra and its primary module}
\label{sec:rho}

\indent
We continue to discuss the symmetric association scheme
${\cal X} = (X,\{R_i\}_{i=0}^d)$ from Section \ref{sec:assoc}. 
Let $\R^X$ denote the vector space over $\R$ consisting of
column vectors with entries in $\R$ and coordinates indexed by $X$.
Observe that $\matX$ acts on $\R^X$ by left multiplication.
For all $y \in X$ let $\hat{y}$ denote the vector in $\R^X$
that has $y$-coordinate $1$ and all other coordinates $0$.
Note that $\{\hat{y} \,|\, y \in X\}$ is a basis for $\R^X$.
Let $T = T(x)$ denote the subalgebra of $\matX$ generated by 
${M}$ and ${M}^*$.
We call $T$ the {\em subconstituent algebra} of $\cal X$ with respect to $x$
\cite[Definition 3.3]{T:subconst1}.
We now describe a certain irreducible $T$-module known as the primary module. 
Let $\bb = \sum_{y \in X} \hat{y}$ denote the ``all $1$'s'' vector in $\R^X$.
By construction, for $0 \leq i \leq d$ we have $E^*_i \bb = A_i \hx$ and 
$A^*_i \bb = |X|E_i \hx$.
Therefore ${M} \hx = {M}^* \bb$. 
Denote this common space by $W$ and observe that $W$ is a $T$-module.
This $T$-module is said to be {\it primary}. 
The $T$-module $W$ is irreducible by \cite[Lemma 3.6]{T:subconst1}.
We now describe two bases for $W$.
For $0 \leq i \leq d$ define
\begin{align}     \label{eq:1i}
 \bb_i &= E^*_i \bb = A_i \hx, & \bb^*_i &= A^*_i \bb = |X|E_i \hx.
\end{align}
Then each of $\{\bb_i\}_{i=0}^d$ and $\{\bb^*_i\}_{i=0}^d$ is a
basis for $W$.
We now describe the transition matrices between these bases.
Using \eqref{eq:Ai} and \eqref{eq:1i} we find 
\begin{align}   \label{eq:b1i}
 \bb_j &= |X|^{-1} \sum_{i=0}^d P_{ij} \bb^*_i & & (0 \leq j \leq d).
\end{align}
By \eqref{eq:b1i} and $PQ=|X|I$,
\begin{align}   \label{eq:b1si}
 \bb^*_j &= \sum_{i=0}^d Q_{ij} \bb_i & & (0 \leq j \leq d).
\end{align}

\medskip

\begin{definition}    \label{def:rho}    \samepage
For all $B \in T$ let $\rho(B) \in \matR$ denote the matrix that
represents $B$ with respect to the basis $\{\bb_i\}_{i=0}^d$.
Thus
\begin{align}       \label{eq:MAj}
 B \bb_j &= \sum_{i=0}^d \rho(B)_{ij} \bb_i  & & (0 \leq j \leq d).
\end{align}
This defines an $\R$-algebra homomorphism $\rho: T \to \matR$.
\end{definition}

\begin{definition}   \label{def:rhos}   \samepage
For all $B \in T$ let $\rho^*(B) \in \matR$ denote the matrix that
represents $B$ with respect to the basis $\{\bb^*_i\}_{i=0}^d$. 
Thus
\begin{align}       \label{eq:MAsj}
 B \bb^*_j &= \sum_{i=0}^d \rho^*(B)_{ij} \bb^*_i  & & (0 \leq j \leq d).
\end{align}
This defines an $\R$-algebra homomorphism $\rho^* : T \to \matR$.
\end{definition}

\begin{lemma}    \label{lem:trans}  \samepage
The following hold for all $B \in T$.
\begin{itemize}
\item[\rm (i)]
 $P \rho(B) P^{-1} = \rho^*(B)$.
\item[\rm (ii)]
 $Q \rho^*(B) Q^{-1} = \rho(B)$.
\end{itemize}
\end{lemma}

\begin{proof}
(i):
By \eqref{eq:b1i} and elementary linear algebra.

(ii):
Follows from (i) and $PQ=|X|I$.
\end{proof}

\begin{lemma}    \label{lem:rhoAi}  \samepage
The following hold for $0 \leq h,i,j \leq d$.
\begin{itemize}
\item[(i)]
$\rho(A_i)$ has $(h,j)$-entry $p^h_{ij}$.
\item[(ii)]
$\rho^*(A^*_i)$ has $(h,j)$-entry $q^h_{ij}$.
\end{itemize}
\end{lemma}

\begin{proof}
(i):
Using \eqref{eq:AiAj} and \eqref{eq:1i} we argue
$A_i \bb_j = A_i A_j \hx = \sum_{h=0}^d p^h_{ij} A_h \hx 
 = \sum_{h=0}^d p^h_{ij} \bb_h$.

(ii):
Similar to the proof of (i).
\end{proof} 

\begin{lemma}     \label{lem:rhoAsi}  \samepage
The following hold for $0 \leq i \leq d$.
\begin{itemize}
\item[\rm (i)]
$\rho(A^*_i)$ is diagonal with $(j,j)$-entry $Q_{ji}$ for 
$0 \leq j \leq d$.
\item[\rm (ii)]
$\rho^*(A_i)$ is diagonal with $(j,j)$-entry $P_{ji}$
for $0 \leq j \leq d$.
\end{itemize}
\end{lemma}

\begin{proof}
(i):
Using \eqref{eq:1i} we argue
$A^*_i \bb_j = A^*_i E^*_j \bb = Q_{ji} E^*_j \bb
 = Q_{ji} \bb_j$.

(ii):
Similar to the proof of (i).
\end{proof}

\begin{lemma}      \label{lem:rhoEsi}   \samepage
The following hold for $0 \leq i \leq d$.
\begin{itemize}
\item[\rm (i)]
$\rho(E^*_i)$ has $(i,i)$-entry $1$ and all other entries $0$.
\item[\rm (ii)]
$\rho^*(E_i)$ has $(i,i)$-entry $1$ and all other entries $0$.
\end{itemize}
\end{lemma}

\begin{proof}
(i):
Using \eqref{eq:1i} we argue
$E^*_i \bb_j = E^*_i E^*_j \bb = \delta_{i,j} E^*_i \bb
 = \delta_{i,j} \bb_i$.

(ii):
Similar to the proof of (i).
\end{proof}

\begin{lemma}     \label{lem:rhoEi}   \samepage
The following hold for $0 \leq h,i,j \leq d$.
\begin{itemize}
\item[\rm (i)]
$\rho(E_i)$ has $(h,j)$-entry $|X|^{-1} Q_{hi} P_{ij}$.
\item[\rm (ii)]
$\rho^*(E^*_i)$ has $(h,j)$-entry $|X|^{-1} P_{hi}Q_{ij}$.
\end{itemize}
\end{lemma}

\begin{proof}
(i):
By Lemma \ref{lem:trans}(ii) 
$\rho(E_i) = Q \rho^*(E_i) Q^{-1}$.
By $PQ=|X|I$ we have $Q^{-1}=|X|^{-1}P$.
By Lemma \ref{lem:rhoEsi}(ii) $\rho^*(E_i)$ has $(i,i)$-entry
$1$ and all other entries $0$.
The result follows from these comments.

(ii):
Similar to the proof of (i).
\end{proof}

\begin{lemma}    \label{lem:Bisym}   \samepage
For $0 \leq i \leq d$ each of $\rho(A_i)$ and $\rho^*(A^*_i)$
is nonnegative and symmetrizable.
\end{lemma}

\begin{proof}
Concerning $\rho(A_i)$, observe
it is nonnegative by Lemma \ref{lem:rhoAi}(i).
Define a diagonal matrix $\Delta \in \matR$ with $(i,i)$-entry 
$\sqrt{k_i}$ for $0 \leq i \leq d$.
Using \eqref{eq:khphij} we routinely find that $\Delta \rho(A_i) \Delta^{-1}$
is symmetric.
Therefore $\rho(A_i)$ is symmetrizable.
The proof for $\rho^*(A^*_i)$ is similar using \eqref{eq:mhqhij} and
Lemma \ref{lem:rhoAi}(ii).
\end{proof}

\medskip

We will need the following well-known facts.

\medskip

\begin{lemma}  {\rm \cite[Proposition III.1.1]{BI}} \label{lem:Ppoly}
\samepage
The following are equivalent.
\begin{itemize}
\item[\rm (i)]
The ordering $\{A_i\}_{i=1}^d$ is $P$-polynomial.
\item[\rm (ii)]
The matrix $\rho(A_1)$ is irreducible tridiagonal.
\item[\rm (iii)]
The graph $\Gamma(\rho(A_1))$ is the bidirected path
$0 \leftrightarrow 1 \leftrightarrow 2 \leftrightarrow \cdots
\leftrightarrow d$.
\end{itemize}
\end{lemma}

\begin{proof}
In the definition of $P$-polynomial,
$p^1_{ij}$ is nonzero if and only if $p^i_{1j}$ is nonzero
by \eqref{eq:khphij} and since $k_h \neq 0$.
Now we obtain the result using Lemma \ref{lem:rhoAi}(ii).
\end{proof}

\begin{lemma}  {\rm \cite[Section III.1]{BI}}  \label{lem:Qpoly}
\samepage
The following are equivalent.
\begin{itemize}
\item[\rm (i)]
The ordering $\{E_i\}_{i=1}^d$ is $Q$-polynomial.
\item[\rm (ii)]
The matrix $\rho^*(A^*_1)$ is irreducible tridiagonal.
\item[\rm (iii)]
The graph $\Gamma(\rho^*(A^*_1))$ is the bidirected path
$0 \leftrightarrow 1 \leftrightarrow 2 \leftrightarrow \cdots
\leftrightarrow d$.
\end{itemize}
\end{lemma}

\begin{proof}
In the definition of $Q$-polynomial,
$q^1_{ij}$ is nonzero if and only if $q^i_{1j}$ is nonzero
by \eqref{eq:mhqhij} and since $m_h \neq 0$.
Now we obtain the result using Lemma \ref{lem:rhoAi}(ii).
\end{proof}

\begin{lemma}   \label{lem:GrhoAi}  \samepage
For $1 \leq i \leq d$ consider the graph $\Gamma(\rho(A_i))$
and $\Gamma(\rho^*(A^*_i))$. 
In either case,
\begin{itemize}
\item[\rm (i)]
$h \rightarrow j$ if and only if $j \rightarrow h$  $(0 \leq h,j\leq d)$.
\item[\rm (ii)]
 $0 \rightarrow i$.
\item[\rm (iii)]
$0 \not\rightarrow h$ if $h \not= i$ $(0 \leq h \leq d)$.
\end{itemize}
\end{lemma}

\begin{proof}
(i):
By Lemmas \ref{lem:AijAji} and \ref{lem:Bisym}.

(ii), (iii):
By Lemma \ref{lem:rhoAi} and since $p_{ih}^{0} = \delta_{h,i}$,
$q_{ih}^{0} = \delta_{h,i}$.
\end{proof}

\begin{lemma}    \label{lem:reltoAi}   \samepage
 For $1 \leq i \leq d$ the following are equivalent.
\begin{itemize}
\item[\rm (i)]
$\cal X$ is $P$-polynomial relative to $A_i$.
\item[\rm (ii)]
The graph $\Gamma(\rho(A_i))$ is a bidirected path.
\end{itemize}
Suppose {\rm (i)} and {\rm (ii)} hold. 
Then the above graph is
$0 \leftrightarrow i \leftrightarrow * \leftrightarrow * \cdots * 
 \leftrightarrow s$,
where $A_s$ is the last adjacency matrix in the $P$-polynomial structure.
\end{lemma}

\begin{proof}
Use Lemmas \ref{lem:Ppoly} and \ref{lem:GrhoAi}.
\end{proof}

\begin{lemma}    \label{lem:reltoAsi}   \samepage
 For $1 \leq i \leq d$ the following are equivalent.
\begin{itemize}
\item[\rm (i)]
$\cal X$ is $Q$-polynomial relative to $E_i$.
\item[\rm (ii)]
The graph $\Gamma(\rho^*(A^*_i))$ is a bidirected path.
\end{itemize}
Suppose {\rm (i)} and {\rm (ii)} hold. 
Then the above graph is
$0 \leftrightarrow i \leftrightarrow * \leftrightarrow * \cdots * 
 \leftrightarrow s$,
where $E_s$ is the last primitive idempotent in the $Q$-polynomial structure.
\end{lemma}

\begin{proof}
Use Lemmas \ref{lem:Qpoly} and \ref{lem:GrhoAi}.
\end{proof}

\section{Proof of Theorems \ref{thm:Q} and \ref{thm:P}}
\label{sec:PQpoly}

\indent
For convenience we first prove Theorem \ref{thm:P}.

\medskip

\begin{proofof}{Theorem \ref{thm:P}}
Fix an ordering $\{A_i\}_{i=1}^d$ of the nontrivial adjacency matrices 
such that $A_1=B$, and let $C=A_s$.
For $0 \leq i \leq d$ the scalar $\th_i = P_{i1}$ is the eigenvalue
of $B$ for $E_i$, and $Q_{si}$ is the dual eigenvalue of $E_i$ for $C$.
For $0 \leq i \leq d$ define a polynomial $f_i \in \R[\lambda]$ by \eqref{eq:fi}.
Let the map $\rho$ be as in Definition \ref{def:rho}. 
By Lemma \ref{lem:Bisym} the matrix $\rho(B)$ is nonnegative,
so we can apply Theorem \ref{thm:mainsym} with $A=\rho(B)$. 
We will do this after a few comments.
Combining Lemma \ref{lem:trans}(i) and Lemma \ref{lem:rhoAsi}(ii) we find
$P\rho(B)P^{-1} = \text{diag}(\th_0,\th_1,\ldots,\th_d)$.
Therefore $\rho(B)$ is multiplicity-free if only if 
$\{\th_i\}_{i=0}^d$ are mutually distinct, and in this case $\rho(E_i)$
is the primitive idempotent of $\rho(B)$ for $\th_i$ $(0 \leq i \leq d)$.
For $0 \leq i \leq d$ the $(s,0)$-entry of $\rho(E_i)$ is given in 
Lemma \ref{lem:rhoEi}(i).
This entry is $|X|^{-1} Q_{si}$ since $P_{i0}=1$.

(i)$\Rightarrow$(ii):
By Lemma \ref{lem:reltoAi} the graph $\Gamma(\rho(B))$ is a bidirected
path with endpoints $s$, $0$.  
Therefore $A=\rho(B)$ satisfies Theorem \ref{thm:mainsym}(i) with $t=0$. 
Applying Theorem \ref{thm:mainsym} we draw two conclusions.
First, $\rho(B)$ is multiplicity-free, so 
$\{\th_i\}_{i=0}^d$ are mutually distinct.
Second, the $(s,0)$-entry of of $\rho(E_i)$ times $f_i(\th_i)$ is 
independent of $i$ for $0 \leq i \leq d$.
By this and our above comments, $f_i(\th_i)Q_{si}$ is independent of $i$ 
for $0 \leq i \leq d$.
Therefore  $f_i(\th_i)Q_{si}=f_0(\th_0)Q_{s0}$ for $0 \leq i \leq d$.
By this and since $Q_{s0}=1$, we find 
$Q_{si}=f_0(\th_0)f_i(\th_i)^{-1}$ for $0 \leq i \leq d$.
In other words, for $0 \leq i \leq d$ the dual eigenvalue of $E_i$ for $C$ 
is equal to \eqref{eq:qid}.

(ii)$\Rightarrow$(i):
Observe that $\rho(B)$ is symmetrizable by Lemma \ref{lem:Bisym}, and
multiplicity-free since $\{\th_i\}_{i=0}^d$ are mutually distinct.
For $0 \leq i \leq d$ the dual eigenvalue of $E_i$ for $C$ is
$Q_{si}$, and this is equal to $f_0(\th_0)f_i(\th_i)^{-1}$ by \eqref{eq:qid}.
By this and our above comments,
for $0 \leq i \leq d$ the $(s,0)$-entry of $\rho(E_i)$ is equal to 
$|X|^{-1} f_0(\th_0)f_i(\th_i)^{-1}$.
So the $(s,0)$-entry of $\rho(E_i)$ times $f_i(\th_i)$ is independent of $i$ 
for $0 \leq i \leq d$, and this common value is nonzero.
Therefore $\rho(B)$ satisfies Theorem \ref{thm:mainsym}(ii) with $t=0$.
Now by Theorem \ref{thm:mainsym},
the graph $\Gamma(\rho(B))$ is a bidirected path with endpoints $s$, $0$.
Now by Lemma \ref{lem:reltoAi} $\cal X$ is $P$-polynomial relative to $B$, 
and $C=A_s$ is the last adjacency matrix in this $P$-polynomial structure.
\end{proofof}

\medskip

The proof of Theorem \ref{thm:Q} is similar to the proof of Theorem \ref{thm:P}.
We give a precise proof for completeness.

\medskip

\begin{proofof}{Theorem \ref{thm:Q}}
Fix an ordering $\{E_i\}_{i=1}^d$ of the nontrivial primitive idempotents
such that $E_1=E$, and let $F=E_s$.
For $0 \leq i \leq d$ the scalar $\th^*_i = Q_{i1}$ is the dual eigenvalue
of $E$ for $A_i$, and $P_{si}$ is the eigenvalue of $A_i$ for $F$.
For $0 \leq i \leq d$ define a polynomial $f^*_i \in \R[\lambda]$ by
\[
 f^*_i=(\lambda-\th^*_0)\cdots(\lambda-\th^*_{i-1})
       (\lambda-\th^*_{i+1})\cdots(\lambda-\th^*_d).
\]
For $0 \leq i \leq d$ let $E^*_i$ (resp. $A^*_i$) denote the dual primitive 
idempotent (resp. dual adjacency matrix) corresponding to $A_i$ (resp. $E_i$). 
Let the map $\rho^*$ be as in Definition \ref{def:rhos}. 
By Lemma \ref{lem:Bisym} the matrix $\rho^*(A^*_1)$ is nonnegative,
so we can apply Theorem \ref{thm:mainsym} with $A=\rho^*(A^*_1)$. 
We will do this after a few comments.
Combining Lemma \ref{lem:trans}(ii) and Lemma \ref{lem:rhoAsi}(i) we find
$Q\rho^*(A^*_1)Q^{-1} = \text{diag}(\th^*_0,\th^*_1,\ldots,\th^*_d)$.
Therefore $\rho^*(A^*_1)$ is multiplicity-free if only if 
$\{\th^*_i\}_{i=0}^d$ are mutually distinct, and in this case $\rho^*(E^*_i)$
is the primitive idempotent of $\rho^*(A^*_1)$ for $\th^*_i$ $(0 \leq i \leq d)$.
For $0 \leq i \leq d$ the $(s,0)$-entry of $\rho^*(E^*_i)$ is given in 
Lemma \ref{lem:rhoEi}(ii). 
This entry is $|X|^{-1} P_{si}$ since $Q_{i0}=1$.

(i)$\Rightarrow$(ii):
By Lemma \ref{lem:reltoAsi} the graph $\Gamma(\rho^*(A^*_1))$ is a bidirected
path with endpoints $s$, $0$.
Therefore $A=\rho^*(A^*_1)$ satisfies Theorem \ref{thm:mainsym}(i) with $t=0$. 
Applying Theorem \ref{thm:mainsym} we draw two conclusions.
First, $\rho^*(A^*_1)$ is multiplicity-free, so 
$\{\th^*_i\}_{i=0}^d$ are mutually distinct.
Second, the $(s,0)$-entry of of $\rho^*(E^*_i)$ times $f^*_i(\th^*_i)$ is 
independent of $i$ for $0 \leq i \leq d$.
By this and our above comments, $f^*_i(\th^*_i)P_{si}$ is independent of $i$ 
for $0 \leq i \leq d$.
Therefore  $f^*_i(\th^*_i)P_{si}=f^*_0(\th^*_0)P_{s0}$ for $0 \leq i \leq d$.
By this and since $P_{s0}=1$, we find 
$P_{si}=f^*_0(\th^*_0)f^*_i(\th^*_i)^{-1}$ for $0 \leq i \leq d$.
In other words, for $0 \leq i \leq d$ the eigenvalue of $A_i$ for $F$ 
is equal to \eqref{eq:pid}.

(ii)$\Rightarrow$(i):
Observe that $\rho^*(A^*_1)$ is symmetrizable by Lemma \ref{lem:Bisym}, and
multiplicity-free since $\{\th^*_i\}_{i=0}^d$ are mutually distinct.
For $0 \leq i \leq d$ the eigenvalue of $A_i$ for $F$ is
$P_{si}$, and this is equal to $f^*_0(\th^*_0)f^*_i(\th^*_i)^{-1}$ by 
\eqref{eq:pid}.
By this and our above comments,
for $0 \leq i \leq d$ the $(s,0)$-entry of $\rho^*(E^*_i)$ is equal to 
$|X|^{-1} f^*_0(\th^*_0)f^*_i(\th^*_i)^{-1}$.
So the $(s,0)$-entry of $\rho^*(E^*_i)$ times $f^*_i(\th^*_i)$ is 
independent of $i$ for $0 \leq i \leq d$, and this common value is nonzero.
Therefore $\rho^*(A^*_1)$ satisfies Theorem \ref{thm:mainsym}(ii) with $t=0$.
Now by Theorem \ref{thm:mainsym},
the graph $\Gamma(\rho^*(A^*_1))$ is a bidirected path with endpoints $s$, $0$.
Now by Lemma \ref{lem:reltoAsi} $\cal X$ is $Q$-polynomial relative to $E$, 
and $F=E_s$ is the last primitive idempotent in this $Q$-polynomial structure.
\end{proofof}

\bigskip\bigskip\noindent
Kazumasa Nomura\\
Professor Emeritus\\
Tokyo Medical and Dental University\\
Kohnodai, Ichikawa, 272-0827 Japan\\
email: knomura@pop11.odn.ne.jp

\bigskip\noindent
Paul Terwilliger\\
Department of Mathematics\\
University of Wisconsin\\
480 Lincoln Drive\\ 
Madison, Wisconsin, 53706 USA\\
email: terwilli@math.wisc.edu

\bigskip\noindent
{\bf Keywords.}
Hessenberg matrix, tridiagonal matrix, association scheme.

\noindent
{\bf 2010 Mathematics Subject Classification}.
05E30, 15A30, 16S50.


\bigskip


{

\small

\begin{thebibliography}{10}



\bibitem{BI}
E. Bannai, T. Ito,
Algebraic Combinatorics I: Association Schemes,
Benjamin/Cummings, London, 1984.

\bibitem{BCN}
A. E. Brouwer, A. M. Cohen, A. Neumaier,
Distance-Regular Graphs,
Springer-Verlag, Berlin, 1989.



\bibitem{KN}
H. Kurihara, H. Nozaki,
An equivalent condition of the $Q$-polynomial property
on the spherical embedding of symmetric association schemes,
preprint;
{\tt arXiv:1007.0473}.

\bibitem{Serre}
D. Serre,
Matrices; Theory and Applications,
Graduate Texts in Mathematics 216,
Springer-Verlag, New York, 2002.


\bibitem{T:subconst1}
P. Terwilliger,
The subconstituent algebra of an association scheme I,
J. Algebraic Combin. 1 (1992) 363--388.




\end{thebibliography}


}
\end{document}